\documentclass[12pt]{article}

\usepackage{amsmath, epsfig, cite, setspace}
\usepackage{amssymb}
\usepackage{amsfonts}
\usepackage{latexsym}
\usepackage{amsthm}

\makeatletter
\renewcommand{\@seccntformat}[1]{{\csname the#1\endcsname}.\hspace{.5em}}
\makeatother

\newtheorem{thm}{Theorem}[section]

\newtheorem{conj}[thm]{Conjecture}
\newtheorem{lem}[thm]{Lemma}

\renewcommand{\qed}{\hfill$\Box$\medskip}
\renewcommand{\thefootnote}{*}

\numberwithin{equation}{section}

\begin{document}

\begin{center}
{\large\bf Proof of a conjecture of Adamchuk}
\end{center}

\vskip 2mm \centerline{Guo-Shuai Mao}
\begin{center}
{\footnotesize $^1$Department of Mathematics, Nanjing
University of Information Science and Technology, Nanjing 210044,  People's Republic of China\\
{\tt maogsmath@163.com  } }
\end{center}

%%date: January 4, 2011
%\vskip 5mm
%\noindent {\it Suggested Running title}: Two Identities of Gould

\vskip 0.7cm \noindent{\bf Abstract.}
In this paper, we prove a congruence which confirms a conjecture of Adamchuk. For any prime $p\equiv1\pmod3$ and $a\in\mathbb{Z}^{+}$, we have
\begin{align*}
\sum_{k=1}^{\frac{2}3(p^a-1)}\binom{2k}k\equiv0\pmod{p^2}.
\end{align*}



\renewcommand{\thefootnote}{**}

\section{Introduction}
In the past decades, many people studied congruences for sums of binomial coefficients (see, for instance, \cite{az-amm-2017,G,gl-arxiv-2019,mao-jnt-2017,mp-jmaa-2017,mw-ijnt-2019,mc-crmath-2019,st-aam-2010,st-ijnt-2011}). In 2011, Sun \cite{st-ijnt-2011} proved that for any odd prime $p$ and $a\in\mathbb{Z}^{+}$,
\begin{align}\label{st}
\sum_{k=0}^{p^a-1}\binom{2k}k\equiv\left(\frac{p^a}3\right)\pmod{p^2},
\end{align}
where $\left(\frac{\cdot}{\cdot}\right)$ is the Jacobi symbol.  Liu and Petrov \cite{lp-aam-2020} showed some congruences on sums of $q$-binomial coefficients.

In 2006, Adamchuk \cite{adamchuk-oeis-2006} conjectured that for any prime $p\equiv1\pmod3$,
\begin{align*}
\sum_{k=1}^{\frac{2}3(p-1)}\binom{2k}k\equiv0\pmod{p^2}.
\end{align*}
Pan and Sun \cite{ps-dm-2006} used a combinatorial identity to deduce that if $p$ is prime then
$$\sum_{k=0}^{p-1}\binom{2k}{k+d}\equiv\left(\frac{p-d}3\right)\pmod{p}\ \ \mbox{for}\ \  d=0,1,\ldots,p.$$
Sun told me he posed the following conjecture which generalizes Adamchuk's conjecture:
\begin{conj}\label{adamsun} Let $p$ be an odd prime and let $a\in\mathbb{Z}^{+}$. If $p\equiv1\pmod 3$ or $2|a$, then
\begin{align*}
\sum_{k=1}^{\frac{2}3(p^a-1)}\binom{2k}k\equiv0\pmod{p^2}.
\end{align*}
\end{conj}
Recall that the Bernoulli numbers $\{B_n\}$ and the Bernoulli polynomials $\{B_n(x)\}$ are defined as follows:
$$\frac x{e^x-1}=\sum_{n=0}^\infty B_n\frac{x^n}{n!}\ \ (0<|x|<2\pi)\ \mbox{and}\ B_n(x)=\sum_{k=0}^n\binom nkB_kx^{n-k}\ \ (n\in\mathbb{N}).$$
Mattarei and Tauraso \cite{mt-jnt-2013} proved that for any prime $p>3$, we have
$$\sum_{k=0}^{p-1}\binom{2k}k\equiv\left(\frac{p}3\right)-\frac{p^2}3B_{p-2}\left(\frac12\right)\pmod{p^2}.$$
The main objective of this paper is to obtain the following result.
\begin{thm}\label{Thadam}
Let $p$ be an odd prime and let $a\in\mathbb{Z}^{+}$. If $p\equiv1\pmod 3$ and $a\in\mathbb{Z}^{+}$, then
\begin{align*}
\sum_{k=1}^{\frac{2}3(p^a-1)}\binom{2k}k\equiv0\pmod{p^2}.
\end{align*}
\end{thm}
In order to prove Theorem \ref{Thadam}, we fist show the following interesting congruence.
\begin{thm}\label{Th3k1} For any prime $p\equiv1\pmod 3$, we have
\begin{align*}
\sum_{\substack{k=0\\k\neq(p-1)/3}}^{(p-1)/2}\frac{\binom{2k}k}{3k+1}\equiv0\pmod{p}.
\end{align*}
\end{thm}
We shall prove Theorem \ref{Th3k1} in Section 2, Section 3 is devoted to prove Theorem \ref{Thadam}.
\section{Proof of Theorem \ref{Th3k1}}
Define the {\it hypergeometric series}
\begin{equation}\label{hypergeometricseries}
{}_{m+1}F_m\bigg[\begin{matrix}
\alpha_0&\alpha_1&\ldots&\alpha_m\\
&\beta_1&\ldots&\beta_m
\end{matrix}\bigg|\,z\bigg]:=\sum_{k=0}^{\infty}\frac{(\alpha_0)_k(\alpha_1)_k\cdots(\alpha_m)_k}{(\beta_1)_k\cdots(\beta_m)_k}\cdot\frac{z^k}{k!},
\end{equation}
where $\alpha_0,\ldots,\alpha_m,\beta_1,\ldots,\beta_m,z\in\mathbb{C}$ and
$$
(\alpha)_k=\begin{cases}\alpha(\alpha+1)\cdots(\alpha+k-1),&\text{if }k\geq 1,\\
1,&\text{if }k=0.\end{cases}
$$
For a prime $p$, let  $\mathbb{Z}_p$ denote the ring of all $p$-adic integers and let $$\mathbb{Z}_p^{\times}:=\{a\in\mathbb{Z}_p:\,a\text{ is prime to }p\}.$$
For each $\alpha\in\mathbb{Z}_p$, define the $p$-adic order $\nu_p(\alpha):=\max\{n\in\mathbb{N}:\, p^n\mid \alpha\}$ and the $p$-adic norm $|\alpha|_p:=p^{-\nu_p(\alpha)}$. Define the $p$-adic gamma function $\Gamma_p(\cdot)$ by
$$
\Gamma_p(n)=(-1)^n\prod_{\substack{1\leq j<n\\ (k,p)=1}}k,\qquad n=1,2,3,\ldots,
$$
and
$$
\Gamma_p(\alpha)=\lim_{\substack{|\alpha-n|_p\to 0\\ n\in\mathbb{N}}}\Gamma_p(n),\qquad \alpha\in\mathbb{Z}_p.
$$
In particular, we set $\Gamma_p(0)=1$. Throughout the whole paper, we only need to use the most basic properties of $\Gamma_p$, and all of them can be found in \cite{Murty02,Robert00}.
For example, we know that
\begin{equation}\label{gamma}
\frac{\Gamma_p(x+1)}{\Gamma_p(x)}=\begin{cases}-x,&\text{if }|x|_p=1,\\
-1,&\text{if }|x|_p>1.
\end{cases}
\end{equation}
\begin{lem}\label{Lem2f1} For any nonnegative integer $n$, we have
\begin{align}\label{2f1}
{}_{2}F_1\bigg[\begin{matrix}
-3n&-3n+\frac12\\
&-4n+\frac23
\end{matrix}\bigg|\,\frac43\bigg]=\frac1{4^n}{}_{2}F_1\bigg[\begin{matrix}
-n&-n+\frac12\\
&-2n+\frac56
\end{matrix}\bigg|\,1\bigg].
\end{align}
\end{lem}
\begin{proof} By using package \texttt{Sigma} due to Schneider \cite{S}, we find that both sides of (\ref{2f1}) satisfy the same recurrence:
$$(3n+2)(6n+1)S[n]-2(12n+1)(12n+7)S(n+1)=0,$$
and it is easy to check that both sides of (\ref{2f1}) are equal for $n=0,1,2$.
\end{proof}
\begin{lem}\label{L} {\rm (\cite{L})}. For any prime $p>3$, we have the following congruences modulo $p$
\begin{align*}H_{\lfloor p/2\rfloor}\equiv-2q_p(2),\ H_{\lfloor p/3\rfloor}\equiv-\frac32q_p(3),\ H_{\lfloor p/6\rfloor}\equiv-2q_p(2)-\frac 32q_p(3).
\end{align*}
\end{lem}
\noindent{\it Proof of Theorem \ref{Th3k1}}.  First for any $\alpha, s\in\mathbb{Z}_p$, we have
$$\frac{\binom{2k}k}{4^k}=\frac{\left(\frac12\right)_k}{(1)_k},\ \  \ \frac{\left(\frac13\right)_k}{\left(\frac43\right)_k}=\frac1{3k+1}\ \ \mbox{and}\ \ (\alpha+sp)_k\equiv(\alpha)_k\pmod p.$$
For each $(p+2)/3\leq k\leq(p-1)/2$, we have
\begin{align*}
\frac{\left(\frac13-\frac{p}6\right)_k}{\left(\frac43-\frac{2p}3\right)_k}&=\frac{\frac{p}6\left(\frac13-\frac{p}6\right)_{(p-1)/3}\left(\frac{p}6+1\right)_{k-(p+2)/3}}{\frac{-p}3\left(\frac43-\frac{2p}3\right)_{(p-4)/3}\left(-\frac{p}3+1\right)_{k-(p-1)/3}}\equiv-\frac12\frac{\left(\frac13\right)_{(p-1)/3}(1)_{k-(p+2)/3}}{\left(\frac43\right)_{(p-4)/3}(1)_{k-(p-1)/3}}\\
&=-\frac12\frac{\left(\frac13\right)_{(p-1)/3}}{\left(\frac43\right)_{(p-4)/3}}\frac1{k-(p-1)/3}\equiv-\frac32\frac{\left(\frac13\right)_{(p-1)/3}}{\left(\frac43\right)_{(p-4)/3}}\frac1{3k+1}\pmod p.
\end{align*}
And
$$
\frac{\left(\frac13\right)_{(p-1)/3}}{\left(\frac43\right)_{(p-4)/3}}=\frac{p-1}3\frac{\left(\frac13\right)_{(p-1)/3}{(p-4)/3!}}{\left(\frac43\right)_{(p-4)/3}{(p-1)/3!}}\equiv-\frac13(-1)^{(p-1)/3}(-1)^{(p-4)/3}=\frac13\pmod p.
$$
Hence for each $(p+2)/3\leq k\leq(p-1)/2$,
$$
\frac{\left(\frac13-\frac{p}6\right)_k}{\left(\frac43-\frac{2p}3\right)_k}\equiv-\frac12\frac1{3k+1}\pmod p.
$$
That means that
$$
\sum_{k=(p+2)/3}^{(p-1)/2}\frac{\left(\frac{1-p}2\right)_k\left(\frac13-\frac{p}6\right)_k}{(1)_k\left(\frac43-\frac{2p}3\right)_k}4^k\equiv-\frac12\sum_{k=(p+2)/3}^{(p-1)/2}\frac{\left(\frac{1}2\right)_k\left(\frac13\right)_k}{(1)_k\left(\frac43\right)_k}4^k\pmod{p}.
$$
So
\begin{align*}
\sum_{\substack{k=0\\k\neq(p-1)/3}}^{(p-1)/2}\frac{\binom{2k}k}{3k+1}&=\sum_{\substack{k=0\\k\neq(p-1)/3}}^{(p-1)/2}\frac{\left(\frac12\right)_k\left(\frac13\right)_k}{(1)_k\left(\frac43\right)_k}4^k\\
&\equiv\sum_{\substack{k=0\\k\neq(p-1)/3}}^{(p-1)/2}\frac{\left(\frac{1-p}2\right)_k\left(\frac13-\frac{p}6\right)_k}{(1)_k\left(\frac43-\frac{2p}3\right)_k}4^k-3\sum_{k=(p+2)/3}^{(p-1)/2}\frac{\left(\frac{1-p}2\right)_k\left(\frac13-\frac{p}6\right)_k}{(1)_k\left(\frac43-\frac{2p}3\right)_k}4^k\pmod p.
\end{align*}
Thus, we only need to prove that
\begin{align}\label{prove}
\sum_{\substack{k=0\\k\neq(p-1)/3}}^{(p-1)/2}\frac{\left(\frac{1-p}2\right)_k\left(\frac13-\frac{p}6\right)_k}{(1)_k\left(\frac43-\frac{2p}3\right)_k}4^k\equiv-\frac32\sum_{k=(p+2)/3}^{(p-1)/2}\frac{\left(\frac{1}2\right)_k\left(\frac13\right)_k}{(1)_k\left(\frac43\right)_k}4^k\pmod p.
\end{align}
Set
$$
\sum_{\substack{k=0\\k\neq(p-1)/3}}^{(p-1)/2}\frac{\left(\frac{1-p}2\right)_k\left(\frac13-\frac{p}6\right)_k}{(1)_k\left(\frac43-\frac{2p}3\right)_k}4^k=\mathfrak{A}-\mathfrak{F},
$$
where
$$
\mathfrak{A}={}_{2}F_1\bigg[\begin{matrix}
\frac{1-p}2&\frac13-\frac{p}6\\
&\frac43-\frac{2p}3
\end{matrix}\bigg|\,4\bigg]
$$
$$\mathfrak{F}=\frac{\left(\frac{1-p}2\right)_{(p-1)/3}\left(\frac13-\frac{p}6\right)_{(p-1)/3}}{(1)_{(p-1)/3}\left(\frac43-\frac{2p}3\right)_{(p-1)/3}}4^{(p-1)/3}.$$
In view of \cite[15.8.1]{olbc-book-2010}, we have
$$
{}_{2}F_1\bigg[\begin{matrix}
a&b\\
&c
\end{matrix}\bigg|\,z\bigg]=(1-z)^{-a}{}_{2}F_1\bigg[\begin{matrix}
a&c-b\\
&c
\end{matrix}\bigg|\,\frac{z}{z-1}\bigg].
$$
Setting $a=\frac{1-p}2, b=\frac13-\frac{p}6, c=\frac43-\frac{2p}3$, we have
$$
\mathfrak{A}=(-3)^{(p-1)/2}{}_{2}F_1\bigg[\begin{matrix}
\frac{1-p}2&1-\frac{p}2\\
&\frac43-\frac{2p}3
\end{matrix}\bigg|\,\frac43\bigg]
$$
Set $n=\frac{p-1}6$ in Lemma \ref{Lem2f1}, $n$ is a nonnegative integer because of $p\equiv1\pmod3$, so we have
$$
{}_{2}F_1\bigg[\begin{matrix}
\frac{1-p}2&1-\frac{p}2\\
&\frac43-\frac{2p}3
\end{matrix}\bigg|\,\frac43\bigg]=\frac1{2^{(p-1)/3}}{}_{2}F_1\bigg[\begin{matrix}
\frac{1-p}6&\frac23-\frac{p}6\\
&\frac76-\frac{p}3
\end{matrix}\bigg|\,1\bigg].
$$
Substituting $m=\frac{p-1}6, b=\frac23-\frac{p}6, c=\frac76-\frac{p}3$ into \cite[15.8.6]{olbc-book-2010}, we have
$$
{}_{2}F_1\bigg[\begin{matrix}
\frac{1-p}6&\frac23-\frac{p}6\\
&\frac76-\frac{p}3
\end{matrix}\bigg|\,1\bigg]=\frac{\left(\frac23-\frac{p}6\right)_{(p-1)/6}}{\left(\frac76-\frac{p}3\right)_{(p-1)/6}}(-1)^{(p-1)/6}{}_{2}F_1\bigg[\begin{matrix}
\frac{1-p}6&\frac{p}6\\
&\frac12
\end{matrix}\bigg|\,1\bigg].
$$
Hence
$$
\mathfrak{A}=(-3)^{(p-1)/2}\frac1{2^{(p-1)/3}}\frac{\left(\frac23-\frac{p}6\right)_{(p-1)/6}}{\left(\frac76-\frac{p}3\right)_{(p-1)/6}}(-1)^{(p-1)/6}{}_{2}F_1\bigg[\begin{matrix}
\frac{1-p}6&\frac{p}6\\
&\frac12
\end{matrix}\bigg|\,1\bigg].
$$
Setting $n=\frac{p-1}6, b=\frac{p}6, c=\frac12$ in \cite[15.4.24]{olbc-book-2010}, we have
\begin{align*}
{}_{2}F_1\bigg[\begin{matrix}
\frac{1-p}6&\frac{p}6\\
&\frac12
\end{matrix}\bigg|\,1\bigg]=\frac{\left(\frac12-\frac{p}6\right)_{(p-1)/6}}{\left(\frac12\right)_{(p-1)/6}}.
\end{align*}
Notice that $(p-1)/2+(p-1)/6=2(p-1)/3$ is even, so
$$
\mathfrak{A}=3^{(p-1)/2}\frac1{2^{(p-1)/3}}\frac{\left(\frac23-\frac{p}6\right)_{(p-1)/6}}{\left(\frac76-\frac{p}3\right)_{(p-1)/6}}\frac{\left(\frac12-\frac{p}6\right)_{(p-1)/6}}{\left(\frac12\right)_{(p-1)/6}}.
$$
Now we calculate the right-side of (\ref{prove}),
\begin{align}\label{p+23}
&\sum_{k=(p+2)/3}^{(p-1)/2}\frac{\left(\frac{1}2\right)_k\left(\frac13\right)_k}{(1)_k\left(\frac43\right)_k}4^k=\sum_{k=(p+2)/3}^{(p-1)/2}\frac{\binom{2k}k}{3k+1}\equiv\sum_{k=(p+2)/3}^{(p-1)/2}\frac{\binom{(p-1)/2}k(-4)^k}{3k+1}\notag\\
&=\sum_{k=0}^{(p-7)/6}\frac{\binom{(p-1)/2}k(-4)^{(p-1)/2-k}}{3((p-1)/2-k)+1}\equiv-2(-1)^{p-1)/2}\sum_{k=0}^{(p-7)/6}\frac{\binom{2k}k}{(6k+1)(16)^k}\notag\\
&=-2(-1)^{p-1)/2}\sum_{k=0}^{(p-7)/6}\frac{\left(\frac12\right)_k\left(\frac16\right)_k}{(1)_k\left(\frac76\right)_k4^k}\equiv-2(-1)^{p-1)/2}\sum_{k=0}^{(p-7)/6}\frac{\left(\frac{1+p}2\right)_k\left(\frac{1-p}6\right)_k}{(1)_k\left(\frac76+\frac{p}3\right)_k4^k}\notag\\
&=-2(-1)^{p-1)/2}(\mathfrak{L}-\mathfrak{Q})\pmod p,
\end{align}
where
$$
\mathfrak{L}={}_{2}F_1\bigg[\begin{matrix}
\frac{1-p}6&\frac12+\frac{p}2\\
&\frac76+\frac{p}3
\end{matrix}\bigg|\,\frac14\bigg],\ \ \ \ \ \mathfrak{Q}=\frac{\left(\frac{1+p}2\right)_{(p-1)/6}\left(\frac{1-p}6\right)_{(p-1)/6}}{(1)_{(p-1)/6}\left(\frac76+\frac{p}3\right)_{(p-1)/6}}\left(\frac14\right)^{(p-1)/6}.
$$
Substituting $a=\frac{1-p}6, b=\frac{1+p}2, c=\frac76+\frac{p}3$ in \cite[15.8.1]{olbc-book-2010}, and then by using \cite[15.4.31]{olbc-book-2010} with $a=\frac{1-p}6$ we have
\begin{align*}
\mathfrak{L}=\left(\frac34\right)^{(p-1)/6}{}_{2}F_1\bigg[\begin{matrix}
\frac{1-p}6&\frac23-\frac{p}6\\
&\frac76+\frac{p}3
\end{matrix}\bigg|\,-\frac13\bigg]=\left(\frac34\right)^{(p-1)/6}\left(\frac89\right)^{(p-1)/3}\frac{\Gamma\left(\frac43\right)\Gamma\left(\frac76+\frac{p}3\right)}{\Gamma\left(\frac32\right)\Gamma\left(1+\frac{p}3\right)}.
\end{align*}
In view of \cite[Lemma 17,(3)]{lr}, we have
$$
\frac{\Gamma\left(\frac43\right)\Gamma\left(\frac76+\frac{p}3\right)}{\Gamma\left(\frac32\right)\Gamma\left(1+\frac{p}3\right)}=\frac3p\frac{\Gamma\left(\frac43\right)\Gamma\left(\frac76+\frac{p}3\right)}{\Gamma\left(\frac32\right)\Gamma\left(\frac{p}3\right)}=-\frac3p\frac{\Gamma_p\left(\frac43\right)\Gamma_p\left(\frac76+\frac{p}3\right)}{\Gamma_p\left(\frac32\right)\Gamma_p\left(\frac{p}3\right)}.
$$
So
\begin{align}\label{l}
\mathfrak{L}=-\frac{34^{(p-1)/3}}{p3^{(p-1)/2}}\frac{\Gamma_p\left(\frac43\right)\Gamma_p\left(\frac76+\frac{p}3\right)}{\Gamma_p\left(\frac32\right)\Gamma_p\left(\frac{p}3\right)}.
\end{align}
Thus, by (\ref{prove}, (\ref{p+23}) and (\ref{l}), we just need to prove that
\begin{align}\label{AFLQ}
\mathfrak{A}-\mathfrak{F}\equiv3(-1)^{(p-1)/2}(\mathfrak{L}-\mathfrak{Q})\pmod p.
\end{align}
By \cite[Lemma 17, (3)]{lr} we know that
\begin{align*}
\mathfrak{A}&=\frac{3^{\frac{p-1}2}}{2^{\frac{p-1}3}}\frac{\left(\frac23-\frac{p}6\right)_{(p-1)/6}}{\left(\frac76-\frac{p}3\right)_{(p-1)/6}}\frac{\left(\frac12-\frac{p}6\right)_{(p-1)/6}}{\left(\frac12\right)_{(p-1)/6}}=\frac6p\frac{3^{\frac{p-1}2}}{2^{\frac{p-1}3}}\frac{\Gamma_p\left(\frac12\right)\Gamma_p\left(\frac13\right)\Gamma_p\left(\frac76-\frac{p}3\right)\Gamma_p\left(\frac12\right)}{\Gamma_p\left(\frac23-\frac{p}6\right)\Gamma_p\left(\frac12+\frac{p}6\right)\Gamma_p\left(-\frac{p}6\right)\Gamma_p\left(\frac13+\frac{p}6\right)}.
\end{align*}
We know that for any $\alpha\in\mathbb{Z}_p$,
\begin{equation}\label{de}
\frac{\Gamma_p'(\alpha)}{\Gamma_p(\alpha)}\equiv \Gamma_p'(0)+H_{p-\langle-\alpha\rangle_p-1}\pmod{p},
\end{equation}
where $H_n=\sum_{k=1}^n\frac1k$ is the $n$th classic harmonic number.

\noindent So we have
$$
p2^{\frac{p-1}3}\mathfrak{A}\equiv6\cdot3^{\frac{p-1}2}\frac{\Gamma_p\left(\frac12\right)\Gamma_p\left(\frac13\right)\Gamma_p\left(\frac76\right)\Gamma_p\left(\frac12\right)}{\Gamma_p\left(\frac23\right)\Gamma_p\left(\frac12\right)\Gamma_p\left(0\right)\Gamma_p\left(\frac13\right)}\left(1-\frac{p}3H_{\frac{p-7}6}+\frac{p}6H_{\frac{p-1}2}\right)\pmod{p^2}.
$$
So by \cite[Definition 4]{lr}, we have
$$
p2^{(p-1)/3}\mathfrak{A}\equiv-(-3)^{(p-1)/2}\frac{\Gamma_p\left(\frac16\right)\Gamma_p\left(\frac13\right)}{\Gamma_p\left(\frac12\right)}\left(1-\frac{p}3H_{(p-1)/6}-2p+\frac{p}6H_{(p-1)/2}\right)\pmod{p^2},
$$
Similarly, we have
$$
p2^{\frac{p-1}3}\mathfrak{F}\equiv-2^{p-1}\frac{\Gamma_p\left(\frac16\right)\Gamma_p\left(\frac13\right)}{\Gamma_p\left(\frac12\right)}\left(1-\frac{p}6H_{\frac{p-1}6}-2p-\frac{5p}6H_{\frac{p-1}3}+\frac{p}2H_{\frac{p-1}2}\right)\pmod{p^2},
$$
$$
3p2^{\frac{p-1}3}(-1)^{\frac{p-1}2}\mathfrak{L}\equiv\frac{2^{p-1}}{(-3)^{\frac{p-1}2}}\frac{\Gamma_p\left(\frac16\right)\Gamma_p\left(\frac13\right)}{\Gamma_p\left(\frac12\right)}\left(1+\frac{p}3H_{\frac{p-1}6}+2p\right)\pmod{p^2},
$$
$$
3p2^{\frac{p-1}3}(-1)^{\frac{p-1}2}\mathfrak{Q}\equiv\frac{\Gamma_p\left(\frac16\right)\Gamma_p\left(\frac13\right)}{\Gamma_p\left(\frac12\right)}\left(1+\frac{2p}3H_{\frac{p-1}3}+\frac{p}3H_{\frac{p-1}6}-\frac{p}2H_{\frac{p-1}2}+2p\right)\pmod{p^2}.
$$
Therefore (\ref{AFLQ}) is equivalent to
\begin{align*}
&-(-3)^{\frac{p-1}2}\left(1-\frac{p}3H_{\frac{p-1}6}-2p+\frac{p}6H_{\frac{p-1}2}\right)+2^{p-1}\left(1-\frac{p}6H_{\frac{p-1}6}-2p-\frac{5p}6H_{\frac{p-1}3}+\frac{p}2H_{\frac{p-1}2}\right)\\
&\equiv\frac{2^{p-1}}{(-3)^{\frac{p-1}2}}\left(1+\frac{p}3+2p\right)-\left(1+\frac{2p}3H_{\frac{p-1}3}+\frac{p}3H_{\frac{p-1}6}-\frac{p}2H_{\frac{p-1}2}+2p\right)\pmod{p^2}.
\end{align*}
By Lemma \ref{L}, we just need to prove that
$$
2^{p-1}-(-3)^{(p-1)/2}-\frac{2^{p-1}}{(-3)^{(p-1)/2}}+1\equiv0\pmod{p^2}.
$$
By using Fermat little theorem and $\left(\frac{-3}p\right)=\left(\frac{p}3\right)=1$, we immediately get that
$$
2^{p-1}-(-3)^{\frac{p-1}2}-\frac{2^{p-1}}{(-3)^{\frac{p-1}2}}+1=\left(2^{p-1}-(-3)^{\frac{p-1}2}\right)\left(1-\frac1{(-3)^{\frac{p-1}2}}\right)\equiv 0\pmod{p^2}.
$$
Therefore the proof of Theorem \ref{Th3k1} is complete.\qed
\section{Proof of Theorem \ref{Thadam}}
\noindent{\it Proof of Theorem \ref{Thadam}}. Now $p\equiv1\pmod 3$, so $\left(\frac{p^a}3\right)=1$, by (\ref{st}) we have
$$
\sum_{k=1}^{p^a-1}\binom{2k}k\equiv0\pmod{p^2}.
$$
Thus we only need to prove that
\begin{equation*}
\sum_{k=(2p^a+1)/3}^{p^a-1}\binom{2k}k\equiv0\pmod{p^2}.
\end{equation*}
Let $k$ and $l$ be positive integers with $k+l=p^a$ and $0<l<p^a/2$. In view of \cite{ps-scm-2014}, we have
\begin{align}\label{l2l}
\frac{l}2\binom{2l}l=\frac{(2l-1)!}{(l-1)!^2}\not\equiv0\pmod{p^a}
\end{align}
and
\begin{align}\label{2k2l}
\binom{2k}k\equiv-p^a\frac{(l-1)!^2}{(2l-1)!}=-\frac{2p^a}{l\binom{2l}l}\pmod{p^2}.
\end{align}
So we have
$$
\sum_{k=(2p^a+1)/3}^{p-1}\binom{2k}k=\sum_{k=1}^{(p^a-1)/3}\binom{2p^a-2k}{p^a-k}\equiv-2p^a\sum_{k=1}^{(p^a-1)/3}\frac1{k\binom{2k}k}\pmod{p^2}.
$$
Hence we only need to show that
\begin{align}\label{p-13}
p^{a-1}\sum_{k=1}^{(p^a-1)/3}\frac1{k\binom{2k}k}\equiv0\pmod{p}.
\end{align}
It is easy to see that for $k=1,2,\ldots,(p^a-1)/2$,
\begin{align}\label{pa-12}
\frac{\binom{(p^a-1)/2}k}{\binom{2k}k/(-4)^k}=\frac{\binom{(p^a-1)/2}k}{\binom{1/2}k}=\prod_{j=0}^{k-1}\frac{(p^a-1)/2-j}{-1/2-j}=\prod_{j=0}^{k-1}\left(1-\frac{p^a}{2j+1}\right)\equiv1\pmod p.
\end{align}
This, with Fermat little theorem yields that
\begin{align*}
p^{a-1}\sum_{k=1}^{(p^a-1)/3}\frac1{k\binom{2k}k}&\equiv p^{a-1}\sum_{k=1}^{(p^a-1)/3}\frac1{k\binom{(p^a-1)/2}k(-4)^k}\equiv-2p^{a-1}\sum_{k=1}^{(p^a-1)/3}\frac1{\binom{(p^a-3)/2}{k-1}(-4)^k}\\
&=\frac12p^{a-1}\sum_{k=0}^{(p^a-4)/3}\frac1{\binom{(p^a-3)/2}k(-4)^k}\pmod p.
\end{align*}
Now we set $n=(p^a-1)/2, m=(p^a-1)/3, \lambda=-\frac14$, then
$$
\sum_{k=0}^{m-1}\frac{\lambda^k}{\binom{n-1}k}=\sum_{k=0}^{n-1}\frac{\lambda^k}{\binom{n-1}k}-\sum_{k=m}^{n-1}\frac{\lambda^k}{\binom{n-1}k}.
$$
So we only need to prove that
\begin{align}\label{n-1m}
p^{a-1}\sum_{k=0}^{n-1}\frac{\lambda^k}{\binom{n-1}k}\equiv p^{a-1}\sum_{k=m}^{n-1}\frac{\lambda^k}{\binom{n-1}k}\pmod p.
\end{align}
In view of \cite{SWZ}, we have
$$
\sum_{k=0}^{n-1}\frac{\lambda^k}{\binom{n-1}k}=n\sum_{k=0}^{n-1}\frac{\lambda^k}{(\lambda+1)^{k+1}}\sum_{i=0}^{n-1-k}\binom{n-1-k}i\frac{(-1)^i}{i+1}+\frac{n\lambda^n}{(\lambda+1)^{n+1}}\sum_{k=0}^{n-1}\frac{(\lambda+1)^{k+1}}{k+1}.
$$
It is easy to show that for each $0\leq k\leq n-1$
$$
\sum_{i=0}^{n-1-k}\binom{n-1-k}i\frac{(-1)^i}{i+1}=\int_0^1\sum_{i=0}^{n-1-k}\binom{n-1-k}i(-x)^idx=\int_0^1(1-x)^{n-1-k}dx=\frac1{n-k}.
$$
Hence
\begin{align*}
\sum_{k=0}^{n-1}\frac{\lambda^k}{\binom{n-1}k}&=n\sum_{k=0}^{n-1}\frac{\lambda^k}{(\lambda+1)^{k+1}(n-k)}+\frac{n\lambda^n}{(\lambda+1)^{n+1}}\sum_{k=0}^{n-1}\frac{(\lambda+1)^{k+1}}{k+1}\\
&=n\sum_{k=1}^{n}\frac{\lambda^{n-k}}{(\lambda+1)^{n-k+1}k}+\frac{n\lambda^n}{(\lambda+1)^{n+1}}\sum_{k=0}^{n-1}\frac{(\lambda+1)^{k+1}}{k+1}\\
&=\frac{n\lambda^n}{(\lambda+1)^{n+1}}\left(\sum_{k=1}^{n}\frac{(\lambda+1)^{k}}{k\lambda^k}+\sum_{k=1}^{n}\frac{(\lambda+1)^{k}}{k}\right).
\end{align*}
In the same way, we have
$$
\sum_{k=m}^{n-1}\frac{\lambda^k}{\binom{n-1}k}=n\sum_{k=0}^{n-1-m}\frac{\lambda^{m+k}}{(\lambda+1)^{k+1}}\sum_{i=0}^{n-1-m-k}\frac{(-1)^i\binom{n-1-m-k}i}{m+i+1}+\frac{n\lambda^n}{(\lambda+1)^{n+1}}\sum_{k=m}^{n-1}\frac{(\lambda+1)^{k+1}}{k+1}.
$$
It is easy to check that for each $0\leq k\leq n-1-m$
\begin{align*}
\sum_{i=0}^{n-1-m-k}\binom{n-1-m-k}i\frac{(-1)^i}{m+i+1}&=\int_0^1\sum_{i=0}^{n-1-m-k}\binom{n-1-m-k}i(-x)^ix^mdx\\
&=\int_0^1x^m(1-x)^{n-1-m-k}dx=B(m+1,n-m-k),
\end{align*}
where $B(P,Q)$ stands for the beta function. It is well known that the beta function relate to gamma function:
$$B(P,Q)=\frac{\Gamma(P)\Gamma(Q)}{\Gamma(P+Q)}.$$
So
$$
B(m+1,n-m-k)=\frac{\Gamma(m+1)\Gamma(n-m-k)}{\Gamma(n-k+1)}=\frac{m!(n-m-k-1)!}{(n-k)!}=\frac1{(m+1)\binom{n-k}{m+1}}.
$$
Therefore
\begin{align*}
\sum_{k=m}^{n-1}\frac{\lambda^k}{\binom{n-1}k}&=\frac{n}{m+1}\sum_{k=0}^{n-1-m}\frac{\lambda^{m+k}}{(\lambda+1)^{k+1}\binom{n-k}{m+1}}+\frac{n\lambda^n}{(\lambda+1)^{n+1}}\sum_{k=m}^{n-1}\frac{(\lambda+1)^{k+1}}{k+1}\\
&=\frac{n}{m+1}\sum_{k=m+1}^{n}\frac{\lambda^{m+n-k}}{(\lambda+1)^{n-k+1}\binom{k}{m+1}}+\frac{n\lambda^n}{(\lambda+1)^{n+1}}\sum_{k=m+1}^{n}\frac{(\lambda+1)^{k}}{k}\\
&=\frac{n\lambda^n}{(\lambda+1)^{n+1}}\left(\frac{\lambda^m}{m+1}\sum_{k=m+1}^{n}\frac{(\lambda+1)^{k}}{\lambda^{k}\binom{k}{m+1}}+\sum_{k=m+1}^{n}\frac{(\lambda+1)^{k}}{k}\right).
\end{align*}
By (\ref{n-1m}), we just need to show that
\begin{equation}\label{nm+1}
p^{a-1}\frac{\lambda^m}{m+1}\sum_{k=m+1}^{n}\frac{(\lambda+1)^{k}}{\lambda^{k}\binom{k}{m+1}}\equiv p^{a-1}\sum_{k=1}^{n}\frac{(\lambda+1)^{k}}{k\lambda^k}+p^{a-1}\sum_{k=1}^{m}\frac{(\lambda+1)^{k}}{k}\pmod p.
\end{equation}
It is obvious that
\begin{align*}
\sum_{k=m+1}^{n}\frac{(\lambda+1)^{k}}{\lambda^{k}\binom{k}{m+1}}=\sum_{k=m+1}^n\frac{(-3)^k}{\binom{k}{m+1}}=\sum_{k=m+1}^n\frac1{\binom{k}{m+1}}\sum_{j=0}^k\binom{k}j(-4)^j=\mathfrak{B}+\mathfrak{C},
\end{align*}
where
$$
\mathfrak{B}=\sum_{j=m+1}^n(-4)^j\sum_{k=j}^n\frac{\binom{k}j}{\binom{k}{m+1}},\ \ \ \ \ \ \ \mathfrak{C}=\sum_{j=0}^{m}(-4)^j\sum_{k=m+1}^n\frac{\binom{k}j}{\binom{k}{m+1}}.
$$
By the following transformation
$$
\frac{\binom{k}j}{\binom{k}{m+1}}=\frac{k!(m+1)!(k-m-1)!}{j!(k-j)!k!}=\frac{(m+1)!(k-m-1)!(j-m-1)!}{j!(k-j)!(j-m-1)!}=\frac{\binom{k-m-1}{j-m-1}}{\binom{j}{m+1}}.
$$
We have
\begin{align*}
\mathfrak{B}=\sum_{j=m+1}^n(-4)^j\sum_{k=j}^n\frac{\binom{k-m-1}{j-m-1}}{\binom{j}{m+1}}=\sum_{j=m+1}^n\frac{(-4)^j}{\binom{j}{m+1}}\sum_{k=0}^{n-j}\binom{k+j-m-1}{j-m-1}.
\end{align*}
By \cite[(1.48)]{g-online}, we have
$$
\mathfrak{B}=\sum_{j=m+1}^n\frac{(-4)^j}{\binom{j}{m+1}}\binom{n-m}{j-m}.
$$
It is easy to show that
$$
\frac{\binom{n-m}{j-m}}{\binom{j}{m+1}}=\frac{(n-m)!(m+1)!(j-m-1)!}{j!(n-j)!(j-m)!}=\frac{n+1}{j-m}\frac{\binom{n}j}{\binom{n+1}{m+1}}.
$$
Thus,
$$\mathfrak{B}=\frac{n+1}{\binom{n+1}{m+1}}\sum_{j=m+1}^n\frac{(-4)^j}{j-m}\binom{n}{j}.$$
Now we calculate $\mathfrak{C}$. First we have the following transformation
$$
\frac{\binom{k}j}{\binom{k}{m+1}}=\frac{k!(m+1)!(k-m-1)!}{j!(k-j)!k!}=\frac{(m+1)!(k-m-1)!(m-j+1)!}{j!(k-j)!(m-j+1)!}=\frac{\binom{m+1}{j}}{\binom{k-j}{m-j+1}}.
$$
Thus,
$$
\mathfrak{C}=\sum_{j=0}^m\binom{m+1}j(-4)^j\sum_{k=m+1}^n\frac1{\binom{k-j}{m-j+1}}=\sum_{j=0}^m\binom{m+1}j(-4)^j\sum_{k=0}^{n-m-1}\frac1{\binom{k+m+1-j}{m-j+1}}.
$$
By using package \texttt{Sigma}, we find the following identity,
$$
\sum_{k=0}^N\frac1{\binom{k+i}i}=\frac{i}{i-1}-\frac{N+1}{(i-1)\binom{N+i}N}.
$$
Substituting $N=n-m-1, i=m+1-j$ into the above identity, we have
$$
\mathfrak{C}=\sum_{j=0}^{m-1}\binom{m+1}j(-4)^j\left(\frac{m+1-j}{m-j}-\frac{n-m}{(m-j)\binom{n-j}{n-m-1}}\right)+(m+1)(-4)^m\sum_{k=1}^{n-m}\frac1k.
$$
It is easy to check that
$$
\frac{(n-m)\binom{m+1}j}{\binom{n-j}{n-m-1}}=\frac{(m+1)!((n-m)!(m+1-j)!}{j!(n-j)!(m+1-j)!}=\frac{(m+1)!((n-m)!}{j!(n-j)!}=\frac{(n+1)\binom{n}j}{\binom{n+1}{m+1}}.
$$
Therefore
$$
\mathfrak{C}=(m+1)\sum_{j=0}^{m-1}\binom{m}j\frac{(-4)^j}{m-j}-\frac{n+1}{\binom{n+1}{m+1}}\sum_{j=0}^{m-1}\binom{n}j\frac{(-4)^j}{m-j}+(m+1)(-4)^m\sum_{k=1}^{n-m}\frac1k.
$$
Hence
\begin{equation*}
\mathfrak{B}+\mathfrak{C}=(m+1)\sum_{j=0}^{m-1}\binom{m}j\frac{(-4)^j}{m-j}+\frac{n+1}{\binom{n+1}{m+1}}\sum_{\substack{j=0\\j\neq m}}^{n}\binom{n}j\frac{(-4)^j}{j-m}+(m+1)(-4)^m\sum_{k=1}^{n-m}\frac1k.
\end{equation*}
That is
\begin{equation}\label{b+c}
\frac{\lambda^m}{m+1}(\mathfrak{B}+\mathfrak{C})=\lambda^m\sum_{j=0}^{m-1}\binom{m}j\frac{(-4)^j}{m-j}+\frac{\lambda^m}{\binom{n}{m}}\sum_{\substack{j=0\\j\neq m}}^{n}\binom{n}j\frac{(-4)^j}{j-m}+H_{n-m}.
\end{equation}
In view of (\ref{pa-12}), we have
\begin{align*}
\sum_{k=1}^n\frac{(-3)^k}k&=\int_0^1\sum_{k=1}^n(-3)^kx^{k-1}dx=-3\int_0^1\sum_{k=0}^{n-1}(-3x)^kdx=-3\int_0^1\frac{1-(-3x)^n}{1+3x}dx\\
&=3\int_0^1\sum_{k=1}^n\binom{n}k(-1)^k(1+3x)^{k-1}dx=\int_1^4\sum_{k=1}^n(-1)^ky^{k-1}dy\\
&=\sum_{k=1}^n\binom{n}k(-1)^k\frac{4^k-1}k\equiv\sum_{k=1}^n\frac{\binom{2k}k}k-\sum_{k=1}^n\binom{n}k\frac{(-1)^k}k\pmod p
\end{align*}
and
\begin{align*}
\sum_{k=1}^n\binom{n}k\frac{(-1)^k}k&=\int_0^1\sum_{k=1}^n\binom{n}k(-1)^kx^{k-1}dx=\int_0^1\frac{(1-x)^n-1}xdx=\int_0^1\frac{y^n-1}{1-y}dy\\
&=-\int_0^1\sum_{k=0}^{n-1}y^kdy=-\sum_{k=0}^{n-1}\frac1{k+1}=-\sum_{k=1}^{n}\frac1{k}.
\end{align*}
In view of \cite[(1.20)]{st-aam-2010}, and by (\ref{l2l}), (\ref{2k2l}) we have
\begin{align}\label{pa-1n}
p^{a-1}\sum_{k=1}^n\frac{\binom{2k}k}k\equiv p^{a-1}\sum_{k=1}^{p^a-1}\frac{\binom{2k}k}k\equiv0\pmod p.
\end{align}
This, with \cite[(1.20)]{st-aam-2010} yields that
$$p^{a-1}\sum_{k=1}^n\frac{(\lambda+1)^k}{k\lambda^k}=p^{a-1}\sum_{k=1}^n\frac{(-3)^k}k\equiv p^{a-1}H_n\pmod p.$$
On the other hand, by \cite[(1.48)]{g-online} we have
\begin{align*}
\sum_{k=1}^m\frac{(\lambda+1)^k-1}{k}&=\sum_{k=1}^m\frac{3^k-1}{k4^k}=\sum_{k=1}^m\frac1k\sum_{j=1}^k\binom{k}j\frac1{(-4)^j}=\sum_{j=1}^m\frac1{j(-4)^j}\sum_{k=j}^m\binom{k-1}{j-1}\\
&=\sum_{j=1}^m\frac1{j(-4)^j}\binom{m}j=\frac1{(-4)^m}\sum_{j=0}^{m-1}\frac{(-4)^j}{m-j}\binom{m}j.
\end{align*}
Hence
$$
\sum_{k=1}^m\frac{(\lambda+1)^k}{k}=\frac1{(-4)^m}\sum_{j=0}^{m-1}\frac{(-4)^j}{m-j}\binom{m}j+H_m.
$$
So modulo $p$ we have
$$
p^{a-1}\sum_{k=1}^m\frac{(\lambda+1)^k}{k}+p^{a-1}\sum_{k=1}^n\frac{(\lambda+1)^k}{k\lambda^k}\equiv p^{a-1}\lambda^m\sum_{j=0}^{m-1}\frac{(-4)^j}{m-j}\binom{m}j+p^{a-1}(H_m+H_n).
$$
It is obvious that
$$
p^{a-1}H_n=p^{a-1}\sum_{k=1}^n\frac1k\equiv p^{a-1}\sum_{j=1}^{(p-1)/2}\frac1{jp^{a-1}}=H_{(p-1)/2}\pmod p
$$
and $p^{a-1}H_m\equiv H_{\lfloor p/3\rfloor}\pmod p$, $p^{a-1}H_{n-m}\equiv H_{\lfloor p/6\rfloor}\pmod p$.\\
This, with (\ref{nm+1}), (\ref{b+c}) and Lemma \ref{L} yields that we only need to prove that
$$
\frac{p^{a-1}}{\binom{n}{m}}\sum_{\substack{j=0\\j\neq m}}^{n}\binom{n}j\frac{(-4)^j}{j-m}\equiv0\pmod p.
$$
Now $p\equiv1\pmod3$, so by \cite[Lemma 17, (2)]{lr}, we can deduce that $p\nmid\binom{n}{m}.$ So we only need to prove that
$$
p^{a-1}\sum_{\substack{j=0\\j\neq m}}^{n}\binom{n}j\frac{(-4)^j}{j-m}\equiv0\pmod p.
$$
It is obvious that
\begin{align}\label{pa-13k1}
p^{a-1}\sum_{\substack{j=0\\j\neq m}}^{n}\binom{n}j\frac{(-4)^j}{j-m}\equiv3p^{a-1}\sum_{\substack{j=0\\j\neq m}}^{n}\binom{n}j\frac{(-4)^j}{3j+1}\pmod p.
\end{align}
In view of (\ref{pa-13k1}), we know that There are only the items $3j+1=p^{a-1}(3k+1)$ with $k=0,1,\ldots,(p-1)/2$ and $k\neq (p-1)/3$, so by Fermat little theorem and Lucas congruence, we have
\begin{align*}
&p^{a-1}\sum_{\substack{j=0\\j\neq m}}^{n}\frac{\binom{n}j(-4)^j}{3j+1}\equiv\sum_{\substack{k=0\\k\neq (p-1)/3}}^{(p-1)/2}\frac{\binom{n}{kp^{a-1}+\frac{p^{a-1}-1}3}(-4)^{kp^{a-1}+\frac{p^{a-1}-1}3}}{3k+1}\\
&\equiv(-4)^{\frac{p^{a-1}-1}3}\binom{\frac{p^{a-1}-1}2}{\frac{p^{a-1}-1}3}\sum_{\substack{k=0\\k\neq (p-1)/3}}^{(p-1)/2}\frac{\binom{n}{k}(-4)^k}{3k+1}\pmod{p}.
\end{align*}
By Theorem \ref{Th3k1}, we immediately get the desired reslut.\\
Therefore the proof of Theorem \ref{Thadam} is complete.\qed

\vskip 3mm \noindent{\bf Acknowledgments.}
The author is funded by the Startup Foundation for Introducing Talent of Nanjing University of Information Science and Technology (2019r062).

\end{document}